\documentclass[reqno]{amsart}
\usepackage{amsmath, amssymb, amsthm, geometry, enumerate, graphicx, amsfonts, hyperref, mathrsfs}
\hypersetup{colorlinks=true,linkcolor=red, anchorcolor=blue, citecolor=blue, urlcolor=red, filecolor=magenta, pdftoolbar=true}
\theoremstyle{plain}
\newtheorem{thm}{\bf Theorem}[section]

\newtheorem{cor}[thm]{\bf Corollary}
\newtheorem{lem}[thm]{\bf Lemma}
\theoremstyle{remark}
\newtheorem{defn}[thm]{\bf Def{}inition}
\newtheorem{rem}[thm]{\bf Remark}
\newtheorem{exa}[thm]{\bf Example}
\numberwithin{equation}{section}
\usepackage{enumerate}
\usepackage{graphicx}
\allowdisplaybreaks
\begin{document}
\baselineskip08pt

\title [Unitary Extension Principle for Non-uniform Wavelet Frames]{Unitary Extension Principle for Nonuniform Wavelet Frames in $L^2(\mathbb{R})$}

\author[H. K. Malhotra]{ Hari Krishan Malhotra }
\address{ \textbf{Hari K. Malhotra}, Department of Mathematics,
University of Delhi, Delhi-110007, India.}
\email{krishan.hari67@gmail.com}

\author[L.  K. Vashisht]{  Lalit  Kumar  Vashisht}
\address{{\bf{Lalit  K. Vashisht}}, Department of Mathematics,
University of Delhi, Delhi-110007, India.}
\email{lalitkvashisht@gmail.com}
\begin{abstract}
 We study the construction of  nonuniform tight wavelet frames for the Lebesgue space $L^2(\mathbb{R})$, where the related translation set is not necessary a group.
The main  purpose of this paper is  to    prove  the unitary extension principle (UEP) and  the  oblique extension principle  (OEP) for construction of  multi-generated nonuniform tight wavelet frames for  $L^2(\mathbb{R})$.  Some examples are also given to illustrate the  results.
\end{abstract}


\keywords{ Hilbert Frame,   Nonuniform wavelet system, Unitary extension principle.\\
	The research of Hari K. Malhotra is supported by  the University Grant Commission (UGC), India (Grant No.: 19/06/2016(i)EU-V).}

\maketitle

\baselineskip15pt

\section{Introduction}
 Wavelets have been extensively studied over the last few years and its role in both pure and applied mathematics is well known. It is not possible to give complete list of applications of wavelets, let us at least mention some \cite{Beno1, BN1, ID, DHRS, Bhan, Hern1, RBCDR, RZ3}, also see  many references therein. Wavelets in $L^2(\mathbb{R})$ are very efficient tools as it gives orthonormal basis for $L^2(\mathbb{R})$ in form of dilation and translation of finite numbers of function in $L^2(\mathbb{R})$  which is very simple and convenient form of basis for $L^2(\mathbb{R})$. Gabardo and Nashed \cite{GN1} considered a generalization of Mallat's classic multiresolution analysis (MRA), which  is based on the theory of spectral pairs.
\begin{defn}\cite[Definition 3.1]{GN1}\label{def1}
Let $N\geq 1$ be a positive integer  and $r$ be an odd integer  relatively prime to  $N$ such that $1\le r \le 2N-1$, an associated nonuniform
multiresolution analysis (abbreviated NUMRA) is a collection $ \{V_j\}_{j \in \mathbb{Z}}$  of closed subspaces of $L^2(\mathbb{R})$ satisfying the following properties:
\begin{enumerate}[$(i)$]
\item $V_j \subset V_{j+1}$ for all $j \in \mathbb{Z}$,
\item $\bigcup_{j \in \mathbb{Z}} V_j$ is dense in $L^2(\mathbb{R})$,
\item $\bigcap_{j \in \mathbb{Z}} V_j = \{0\}$,
\item $f(x) \in V_j$ if and only if $f(2Nx) \in V_{j+1}$,
\item There exists a function $\phi \in V_0$, called the \emph{scaling function}, such
that the collection \break $\{\phi(x - \lambda)\}_{\lambda \in \Lambda}$, where $\Lambda =\left\{0,\frac{r}{N}  \right\} +2\mathbb{Z}$, is a complete
orthonormal system for $V_0$.
\end{enumerate}
\end{defn}
Here, the translate set  $\Lambda =\left\{0,\frac{r}{N}  \right\} +2\mathbb{Z} $ may   not be a group. One may observe  that the standard definition of a one-dimensional multiresolution analysis with dilation factor equal to $2$ is a special case of  NUMRA given in Definition \ref{def1}. Related to the one-dimensional spectral pairs, Gabardo and  Yu \cite{GN2} considered sets of   nonuniform wavelets  in $L^2(\mathbb{R})$. For fundamental properties of nonuniform wavelets based on the  spectral pair, we refer to \cite{GN1,GN2, YuG}.

Ron and Shen \cite{Ron1} introduced the unitary extension principle which gives the construction of a multi-generated tight wavelet frame  for  $L^2(\mathbb{R}^d)$, based on a given refinable function. Tight wavelet frames gives more convenient way to represent a function in $L^2(\mathbb{R})$ in comparison of non-tight wavelet frames as in that case frame operator is constant multiple of identity operator in $L^2(\mathbb{R})$. Christensen and Goh in  \cite{OCG1} generalized  the unitary extension principle to  locally compact abelian groups. They gave general constructions, based on B-splines on the group itself as well as on characteristic functions on the dual group.  Motivated by the work of  Gabardo and Nashed \cite{GN1} for the construction of nonuniform wavelets, and application of frames in applied and pure mathematics, we study nonuniform wavelet frames for the Lebesgue space  $L^2(\mathbb{R})$. Notable contribution in the paper is   to introduce the  unitary extension principle for the construction of    multi-generated tight nonuniform wavelet frames of the form
\begin{align*}
\{\Psi_{j, \lambda, \ell} \}_{j\in \mathbb{Z},\lambda \in \Lambda \atop \ell =1,2,\dots,n} = \Big\{(2N)^{\frac{j}{2}}\psi_{1}(2N)^j\gamma-\lambda\Big\}_{j\in \mathbb{Z},\lambda \in \Lambda} \bigcup \cdots \bigcup \Big\{(2N)^{\frac{j}{2}}\psi_{n}(2N)^j\gamma-\lambda\Big\}_{j\in \mathbb{Z},\lambda \in \Lambda}
\end{align*}
in  $L^2(\mathbb{R})$.

\subsection{Overview and main results}
The paper  is  organized as follows. In Section  \ref{sec2},  we give  basic notations, definitions and properties of operators related with nonuniform wavelet frames in $L^2(\mathbb{R})$. The general setup for  nonuniform wavelet frame system  in $L^2(\mathbb{R})$  is  also given in Section  \ref{sec2}.  Section \ref{sec3}  gives some auxiliary  results  needed in the rest of the paper. The  main results are given in  Section~\ref{sec4}. Theorem \ref{thm4.1} gives the  unitary extension principle (UEP) for  the construction of    multi-generated tight nonuniform wavelet frames for $L^2(\mathbb{R})$. The  extended version of UEP (or oblique extension principle) for nonuniform wavelet frames  for $L^2(\mathbb{R})$ can be found in Theorem \ref{thm4.2}.   Some examples are given in  Section \ref{sec5} to illustrate our results.

\subsection{Relation to existing work  and motivation}
 Duffin and  Schaeffer \cite{DS} introduced  the concept of a frame  for separable  Hilbert spaces, while  addressing some difficult problems from the theory of nonharmonic analysis. Let $\mathcal{H}$ be an infinite dimensional  separable Hilbert space  with inner product $\langle.,.\rangle$. The  norm induced by the inner product $\langle.,.\rangle$ is given by $\|f\| = \sqrt{\langle f, f \rangle}$, $f \in \mathcal{H}$. A  family  $\{f_k\}_{k=1}^{\infty} \subset \mathcal{H}$  is called a \emph{frame}  for   $\mathcal{H}$,   if there exist  positive scalars $A_o \leq B_o < \infty$ such that  for all $  f \in \mathcal{H}$,
\begin{align}\label{eq1.1e}
A_o \|f\|^2\leq  \sum_{k=1}^{\infty}  |\langle f, f_k\rangle|^2 \leq B_o \|f\|^2.
\end{align}
The scalars $A_o$ and $B_o$ are called \emph{lower frame bound} and \emph{upper frame bound}, respectively.  If  it is possible to choose $A_o = B_o$, then we say that   $\{f_k\}_{k=1}^{\infty} $ is a  \emph{$A_o$-Parseval frame} (or \emph{$A_o$-tight frame}); and \emph{Parseval frame} if $A_o = B_o =1$. If only upper inequality in \eqref{eq1.1e} holds, then we say that  $\{f_k\}_{k=1}^{\infty} $ is   a \emph{Bessel sequence}  sequence  with \emph{Bessel bound} $B_o$.  If  $\{f_k\}_{k=1}^{\infty}$ is a frame for $\mathcal{H}$, then  $S : \mathcal{H}\rightarrow \mathcal{H}$ given  by $ S f = \sum_{k=1}^{\infty} \langle f, f_k\rangle f_k$   is  a    bounded, linear and invertible  on $\mathcal{H}$, and is called the frame operator. This gives the \emph{reconstruction formula} of each vector  $f \in \mathcal{H}$,
\begin{align*}
f = SS^{-1}f =\sum_{k=1}^{\infty} \langle S^{-1}f, f_k \rangle f_k.
\end{align*}
Thus, each vector has an explicit series  expansion which need not be unique.   For application of frames in both pure and applied mathematics,  we refer to book of  Casazza and Kutyniok  \cite{CK},  Christensen \cite{OC2} and Han \cite{Bhan}. Nowadays the theory of iterated function systems, quantum mechanics and wavelets is emerging in important applications in frame theory, see \cite{D, VD3,  RZ1}  and many references therein.
Very recent work on  discrete frames of translates and discrete wavelet frames, and their duals in  finite dimensional spaces can be found in  \cite{DLKV, DV2}.
Wavelet frames in $L^2(\mathbb{R})$ are also very powerful tool for representing functions  in $L^2(\mathbb{R})$ as sum of series of functions  which are dilation and translation of finite number of functions  in $L^2(\mathbb{R})$. It provides us convenient tool to expansion of functions in $L^2(\mathbb{R})$ of similar type as one that arise in orthonormal basis, however,  wavelet frame conditions are weaker that makes wavelet frame more flexible. Nonuniform wavelet frames could be used in signal processing, sampling theory, speech recognition and various other areas, where instead of integer shifts nonuniform shifts are needed. In \cite{SPman}, Sharma and Manchanda  gave necessary and sufficient conditions for   nonuniform wavelet frames in $L^2(\mathbb{R})$.

 Motivated by the work of Gabardo and Nashed \cite{GN1} and  Gabardo and  Yu \cite{GN2} in the study of  nonuniform wavelets, we study frame properties of nonuniform wavelets in the Lebesgue space $L^2(\mathbb{R})$.  We recall that the extension problems in frame theory has a long history. It is showed in \cite{CKKm} that the  extension problem has a solution in the sense that ``any Bessel sequence can be extended to a tight frame by adjoining a suitable family of vectors in the underlying space.''
Ron and Shen introduced unitary extension principle for construction of tight wavelet frames in the Lebesgue space $L^2(\mathbb{R}^d)$. The unitary extension principle  allows construction of tight wavelet frames with compact support, desired smoothness; and good approximation of functions.  In real life application all signals are not obtained from  uniform shifts; so there is a natural question  regarding  analysis and decompositions of  this types of signals by a stable mathematical tool.   Gabardo and Nashed \cite{GN1} and  Gabardo and  Yu \cite{GN2} filled this gap by  the concept of  nonuniform multiresolution analysis. In the direction of  construction of Parseval frames from nonuniform multiwavelets systems, we develop  a  general setup  and prove  the unitary extension principle  for construction of  multi-generated nonuniform tight wavelet frames for  $L^2(\mathbb{R})$. Ron and Shen \cite{Ron1} gave  the unitary extension principle, where  conditions for the construction of  multi-generated tight wavelet frames for the Lebesgue space  $L^2(\mathbb{R}^d)$  are based on a given refinable function.
%The conditions in our general set up and in the unitary extension principle are different than that given by Ron and Shen, see  Remark \ref{remd2.2f9} for  details.

\section{Preliminaries} \label{sec2}
 As is standard,  $\mathbb{Z}$, $\mathbb{N}$ and  $\mathbb{R}$ denote  the set of all  integers, positive integers and  real numbers, respectively.
Throughout the paper,  $N \in \mathbb{N}$, $r $ be an odd integer relative prime to $N$ such that $1~\le~r \le~2N-1$ and  $\Lambda =\left\{0,\frac{r}{N}  \right\} +2\mathbb{Z} $. Notice that the discrete set $\Lambda$ is not always a group.  The support of a function $\psi$ is  denoted by  Supp $\psi$, and defined as  Supp $\psi = \text{clo}\Big(\{x: \psi(x) \ne 0\}\Big)$. Symbol $\overline{z}$ denote the complex conjugate of a complex number $z$. The conjugate transpose of a matrix $H$ is denoted by $H^{*}$. The characteristic function of a set $E$ is denoted by $\chi_{E}$. The spaces $L^2(\mathbb{R})$ and $L^{\infty}(\mathbb{R})$  denote the  equivalence classes of square-integrable functions and essentially bounded functions on $\mathbb{R}$, respectively. Next, we recall the Parseval identity.  Let $\{e_k\}_{k \in \mathbb{Z}}$  be an orthonormal basis for a Hilbert space $\mathcal{H}$.  Then,
	\begin{align*}
	\sum \limits_{k\in \mathbb{Z}} |\langle f,e_k \rangle|^2=\|f\|^2,  \ f\in \mathcal{H} \quad \text{(Parseval identity)}.
	\end{align*}
For $a, b \in \mathbb{R}$, we consider the following operators on  $L^2({\mathbb{R}})$.
\begin{align*}
&  T_a:L^2({\mathbb{R}})\to L^2({\mathbb{R}}),  \quad T_af(\gamma)=f(\gamma-a)  \quad (\text{Translation by} \ a ),\\
&  E_b:L^2({\mathbb{R}})\to L^2({\mathbb{R}}), \quad  E_b f(\gamma)=e^{2\pi i b \gamma}f(\gamma) \quad (\text{Modulation by} \ b), \\
& L: L^2({\mathbb{R}})\to L^2({\mathbb{R}}), \quad Lf(\gamma)=\sqrt{2N}f(2N\gamma) \quad (\text{N-Dilation operator}).
\end{align*}
The $j$ fold $N$-dilation, where $j\in \mathbb{Z}$,  is given by
\begin{align*}
L^jf(\gamma)=(2N)^{\frac{j}{2}}f((2N)^j\gamma).
\end{align*}

\begin{defn}
Let $\{\psi_1,\psi_2,\dots,\psi_n\} \subset   L^2(\mathbb{R})$ be a finite set, where  $\psi_{\ell} \ne 0$, $1 \leq \ell \leq n$.   The family
\begin{align*}
\{L^jT_{\lambda}\psi_{\ell} \}_{j\in \mathbb{Z},\lambda \in \Lambda \atop \ell =1,2,\dots,n} = \Big\{(2N)^{\frac{j}{2}}\psi_{1}(2N)^j\gamma-\lambda\Big\}_{j\in \mathbb{Z},\lambda \in \Lambda} \bigcup \cdots \bigcup \Big\{(2N)^{\frac{j}{2}}\psi_{n}(2N)^j\gamma-\lambda\Big\}_{j\in \mathbb{Z},\lambda \in \Lambda}
\end{align*}
is called a  nonuniform wavelet frame  for $L^2(\mathbb{R})$,  if   there exist finite  positive   constants $A$ and $B$ such that
\begin{align*}
A\|f\|^2 \le  \sum_{j\in \mathbb{Z}} \sum_{\lambda \in \Lambda}\sum_{\ell=1}^{n}  |\langle f,L^jT_\lambda \psi_{\ell} \rangle|^2 \le B\|f\|^2 \ \text{for all} \ f \in L^2(\mathbb{R}).
\end{align*}
	
\end{defn}

The  Fourier transform of a function $f$ is denoted by $\mathcal{F}f \ \text{or}\  \hat{f}$, and  defined as
\begin{align*}
\mathcal{F}f=\hat{f}(\gamma)=\int\limits_{-\infty}^{\infty}f(x)e^{- 2\pi ix \gamma} dx.
\end{align*}

For  $N \in \mathbb{N}$, $j \in \mathbb{Z}$ and $a \in \mathbb{R}$, by direct calculation, we have the following properties.
\begin{enumerate}[$(i)$]
\item 	$L^j:L^2(\mathbb{R})\to L^2(\mathbb{R})$ is unitary map.
\item $L^jT_a =T_{(2N)^{-j}a}L^j$.
\item $\mathcal{F}L^j=L^{-j}\mathcal{F}$.
\item $\mathcal{F}T_a=E_{-a}\mathcal{F}$.
\end{enumerate}

\section{The Nonuniform General Setup}
In this section, we give  a  list of  assumptions which will be used in the construction of Parseval nonuniform wavelet frames. To be precise,  in formulation of the unitary extension principle there is long list of assumption, instead of writing each assumption again and again, we state all assumptions at once and call it \emph{nonuniform general setup}: Let $\psi_0\in L^2(\mathbb{R})$  be such that
\begin{enumerate}[$(i)$]
	\item $\hat{\psi_0}(2N\gamma)=H_0(\gamma)\hat{\psi_0}(\gamma), \  H_0(\gamma)\in L^{\infty}(\mathbb{R})$;
	\item Supp $ {\hat{\psi_0}}(\gamma)\subseteq [0,\frac{1}{4N}]$; and
	\item $\lim \limits_{\gamma \to 0^+}\hat{\psi_0}(\gamma)=1$.
\end{enumerate}
Further,  let $H_1,H_2,\dots,H_n \in L^{\infty}(\mathbb{R})$,  and define $\psi_1,\psi_2,\dots,\psi_n \in L^2(\mathbb{R})$ such that
\begin{align*}
\widehat{\psi}_{\ell}(2N\gamma)= H_{\ell} (\gamma)\hat{\psi_0}(\gamma), \   \ell = 1,2, \dots,  n.
\end{align*}
Let  $H(\gamma)$ be  a $(n+1) \times 1$ matrix given by
\begin{align*}
H(\gamma)=\begin{bmatrix}
H_0(\gamma)\\
H_1(\gamma)\\
\vdots\\
H_n(\gamma)
\end{bmatrix}_{(n+1) \times 1}.
\end{align*}
Then, the  collection  $\{\psi_{\ell}, H_{\ell}\}_{\ell=0}^n$ is called  the nonuniform  general setup.

\section{Some Auxiliary Results}\label{sec3}
In this section, we give some  auxiliary results that will be used in the sequel.
\begin{lem}\label{lem00}
For any  $f \in L^1(\mathbb{R})$, the function $\mathcal{S}f (\gamma)=\sum\limits_{k\in \mathbb{Z}} f(\gamma + Nk)$ is well defined, $N$-periodic and belongs to $L^1(0,N)$.
\end{lem}
\begin{proof}
It is clear that $\mathcal{S}f (\gamma)=\sum\limits_{k\in \mathbb{Z}} f(\gamma + Nk)$ is $N$-periodic.
For any  $f \in L^1(\mathbb{R})$, we have
\begin{align*}
\int\limits_{0}^{N} \sum\limits_{k\in \mathbb{Z}}|f(\gamma +Nk)|d\,\gamma=\int\limits_{\mathbb{R}} |f(\gamma)|\, d\gamma <\infty.
\end{align*}
Thus, $\mathcal{S}f (\gamma)$ is well defined a.e. on $\mathbb{R}$, and also  belongs to $L^1(0,N)$.
\end{proof}

\begin{lem}\label{lem2.1h}
Assume that
\begin{enumerate}[$(i)$]
\item  $\psi_0 \in L^2\mathbb{(R)}$,  $\lim\limits_{\gamma \to 0^+} \hat{\psi_0}(\gamma) =1$  and  Supp $\hat{\psi_0} (\gamma)$ $\subseteq [0,\frac{1}{N}]$;
\item   $f \in L^2 \mathbb{(R)}$  such that $\hat{f} \in C_c{\mathbb{(R)}}$.
\end{enumerate}
  Then,  for any $\epsilon > 0 $ there exist $J \in \mathbb{Z}$ such that
\begin{align*}
(1-\epsilon)\|f\|^2 \le\sum\limits_{\lambda \in \Lambda} |\langle f,L^jT_\lambda \psi_0 \rangle|^2 \le(1+\epsilon) \|f\|^2 \ \text{for all} \  j \ge J.
\end{align*}
\end{lem}

\begin{proof}
For any  $j \in \mathbb{Z}$, $(L^j\hat{f})\bar{\hat{\psi_0}}\in L^1(\mathbb{R})$. Therefore,   by Lemma \ref{lem00}, the function  $ \mathcal{S}(L^j\hat{f})\bar{\hat{\psi_0}}$ is well defined.
 Further, for $\gamma \in [0,N]$, we have
	\begin{align*}
\mathcal{S}(L^j\hat{f})\bar{\hat{\psi_0}} &= \sum\limits_{k\in \mathbb{Z}}((L^j\hat{f})\bar{\hat{\psi_0}})(\gamma-Nk)\\
&=\sum\limits_{k\in \mathbb{Z}} (L^j\hat{f})(\gamma-Nk) \bar{\hat{\psi_0}}(\gamma-Nk).
\end{align*}
Thus,   $ \mathcal{S}(L^j\hat{f})\bar{\hat{\psi_0}}$ is bounded by finite linear combinations of translates of $\bar{\hat{\psi_0}}$	and  \break $ \mathcal{S}(L^j\hat{f})\bar{\hat{\psi_0}} \in L^2[0,N]$.

Note that
\begin{align*}
\langle f,L^jT_\lambda \psi_0 \rangle =\langle \widehat{f},\widehat{L^jT_\lambda \psi_0} \rangle
 =\langle\widehat{f},L^{-j}E_{-\lambda}\hat{\psi_0}\rangle
 =\langle L^j\hat{f},E_{-\lambda}\hat{\psi_0} \rangle.
\end{align*}
Using  Supp $\hat{\psi}_0(\gamma)\subseteq\left[0,\frac{1}{4N}\right]$ and $\frac{1}{4N}<\frac{1}{2}$, we compute
\begin{align}\label{eqrf1}
&\sum\limits_{\lambda \in \Lambda}	|\langle f,L^jT_\lambda \psi_0 \rangle|^2 \notag\\
 &=\sum\limits_{\lambda \in \Lambda}	|\langle L^j\hat{f},E_{-\lambda}\hat{\psi_0} \rangle  |^2 \notag\\
&=\sum\limits_{\lambda\in 2\mathbb{Z}}	|\langle L^j\hat{f},E_{-\lambda}\hat{\psi_0} \rangle  |^2 +\sum\limits_{\lambda \in ( \frac{r}{N}+2\mathbb{Z})}	|\langle L^j\hat{f},E_{-\lambda}\hat{\psi_0} \rangle  |^2 \notag\\
&=\sum\limits_{m\in\mathbb{Z}} \Big| \int\limits_{0}^{N}\mathcal{S}((L^j\hat{f})  \bar{\hat{\psi_0}} )(\gamma) e^{2\pi i(2m)\gamma} \,d\gamma  \Big|^2 +\sum\limits_{m\in\mathbb{Z}} \Big| \int\limits_{0}^{N}\mathcal{S}((L^j\hat{f})  \bar{\hat{\psi_0}} )(\gamma) e^{2\pi i(\frac{r}{N}+2m)\gamma} \,d\gamma  \Big|^2 \notag \\
&=\sum\limits_{m\in\mathbb{Z}} \Big| \int\limits_{0}^{\frac{1}{4N}}((L^j\hat{f})  \bar{\hat{\psi_0}} ) e^{2\pi i(2m)\gamma} \,d\gamma  \Big|^2 +\sum\limits_{m\in\mathbb{Z}} \Big| \int\limits_{0}^{\frac{1}{4N}}((L^j\hat{f})  \bar{\hat{\psi_0}} ) e^{2\pi i(\frac{r}{N}+2m)\gamma} \,d\gamma  \Big|^2 \notag\\
&=\sum\limits_{m\in\mathbb{Z}} \Big| \int\limits_{0}^{\frac{1}{2}}((L^j\hat{f})  \bar{\hat{\psi_0}} ) e^{2\pi i(2m)\gamma} \,d\gamma  \Big|^2 +\sum\limits_{m\in\mathbb{Z}} \Big| \int\limits_{0}^{\frac{1}{2}}((L^j\hat{f})  \bar{\hat{\psi_0}} ) e^{2\pi i(\frac{r}{N}+2m)\gamma} \,d\gamma  \Big|^2
\end{align}
Applying the Parseval identity  on $L^2(0,\frac{1}{2})$ with respect to an orthonormal bases
$\{\sqrt{2}e^{2 \pi i(2m)\gamma}\}$  in \eqref{eqrf1}, we obtain
\begin{align}\label{eqrev1}
\sum\limits_{\lambda \in \Lambda}	|\langle f,L^jT_\lambda \psi_0 \rangle|^2
&=\frac{1}{2}\int\limits_{0}^{\frac{1}{2}}|(L^j\hat{f})  \bar{\hat{\psi_0}}|^2\,d\gamma +\frac{1}{2}\int\limits_{0}^{\frac{1}{2}}|(L^j\hat{f})  \bar{\hat{\psi_0}}|^2\,d\gamma	\notag\\
&=\int\limits_{0}^{\frac{1}{2}}|(L^j\hat{f})  \bar{\hat{\psi_0}}|^2\,d\gamma.
\end{align}

Let $\epsilon >0$ be given. Since $\hat{\psi_0}(\gamma) \to 1$ as $\gamma \to 0^+ $, we can choose  $b \in ]0,\frac{1}{2}[ $ so that
\begin{equation}\label{eq1f}
(1-\epsilon)\le |\hat{\psi_0}(\gamma)|^2 \le (1+\epsilon),  \  \text{where} \  0<\gamma<b.
\end{equation}

Choose $J \in \mathbb{Z}$ large enough,  so that Supp $(L^j\hat{f}) \subset [-b,b]$ for all $j \geq J$. Then, by \eqref{eqrev1}, we have
\begin{equation}\label{eq2f}
\sum\limits_{\lambda \in \Lambda}	|\langle f,L^jT_\lambda \psi_0 \rangle|^2=
\int\limits_{0}^{b}|(L^j\hat{f})  \bar{\hat{\psi_0}}|^2\,d\gamma \  \text{for all} \ j \geq J.
\end{equation}

By \eqref{eq1f}, \eqref{eq2f}   and the fact that $L^j$ is unitary map, we have
\begin{align*}
(1-\epsilon)\|\hat{f}\|^2 \le \sum\limits_{\lambda \in \Lambda}	|\langle f,L^jT_\lambda \psi_0 \rangle|^2 \le (1+\epsilon)\|\hat{f}\|^2 \ \text{ for all} \  j \ge J.
\end{align*}
 Since the  Fourier transform is  unitary map, we get
\begin{align*}
(1-\epsilon)\|{f}\|^2 \le \sum\limits_{\lambda \in \Lambda}	|\langle f,L^jT_\lambda \psi_0 \rangle|^2 \le (1+\epsilon)\|{f}\|^2 \ \text{ for all} \  j \ge J.
\end{align*}
This concludes the proof.
\end{proof}

\begin{lem}\label{lem2.2}
Suppose that
\begin{enumerate}[$(i)$]
\item  $\psi_0 \in L^2(\mathbb{R})$ satisfies  Supp $\hat{\psi_0} \subseteq [0,\frac{1}{4N}]$ and
$\hat{\psi_0}(2N\gamma)=H_0(\gamma)\hat{\psi_0}(\gamma)$,where $H_0(\gamma) \in L^\infty (\mathbb{R})$;
\item  $f \in L^2(\mathbb{R})$ with $\hat{f} \in C_c(\mathbb{R})$,  and $ H_1,H_2,\dots,H_n \in  L^\infty (\mathbb{R})$  such that  the $(n+1)\times 1$ matrix
$$H(\gamma)=
\begin{bmatrix}
H_0(\gamma)\\
H_1(\gamma)\\
\vdots\\
H_n(\gamma)
\end{bmatrix}_{(n+1)\times 1}$$
satisfies ${H(\gamma)}^\ast H(\gamma)=1$ $a.e.$;
\item $\psi_1,\psi_2, \dots \psi_n \in L^2(\mathbb{R})$ such that $\hat{\psi}_{\ell}(2N\gamma)=H_{\ell}(\gamma)\hat{\psi_0}(\gamma),  \  \ell=1,2,\dots n$.
\end{enumerate}
Then
\begin{align*}
\sum \limits_{\ell=0}^n\sum\limits_{\lambda \in \Lambda} |\langle f,L^{j-1}T_\lambda \psi_{\ell} \rangle|^2 =\sum\limits_{\lambda \in \Lambda} |\langle f,L^jT_\lambda \psi_0 \rangle|^2.
\end{align*}
\end{lem}	

\begin{proof}
For any  $ j\in \mathbb{Z}$ and for any  $\ell=0,1,\dots n$, we  have
\begin{align}
 \langle f,L^{j-1}T_\lambda \psi_{\ell} \rangle &=\langle L^{-j}f,L^{-1}T_\lambda \psi_{\ell} \rangle \notag\\
&=\langle L^{-j}f,T_{(2N)\lambda}L^{-1}\psi_{\ell} \rangle \notag\\
&=\langle L^j\hat{f},E_{-(2N)\lambda}L\hat\psi_\ell\rangle \notag\\
&=\int \limits_{\mathbb{R}}(L^j\hat{f})(\gamma)\sqrt{2N}\overline{\hat{\psi_{\ell}}(2N\gamma)}e^{2\pi i(2N\lambda)\gamma} d\gamma \notag\\
&=\sqrt{2N}\int \limits_{\mathbb{R}}(L^j\hat{f})(\gamma)\overline{H_{\ell}(\gamma)}\overline{\hat{\psi_0}(\gamma)}e^{2\pi i(2N\lambda)\gamma} d\gamma. \label{eq3.4rev}
\end{align}
Using Supp $\hat{\psi_0}\subseteq [0,\frac{1}{4N}]$, and  Parseval identity  on $ L^2(0,\frac{1}{4N})$ with respect to orthonormal basis $\{2\sqrt{N} e^{2\pi i (4Nm)\gamma}\}_{m \in \mathbb{Z}}$, we have
\begin{align}\label{eq3.5extr}
\sum\limits_{\lambda \in \Lambda} |\langle f,L^{j-1}T_\lambda \psi_{\ell} \rangle|^2 &=\sum\limits_{\lambda \in 2\mathbb{Z}} |\langle f,L^{j-1}T_\lambda \psi_{\ell} \rangle|^2+\sum\limits_{\lambda \in (\frac{r}{N}+2\mathbb{Z})} |\langle f,L^{j-1}T_\lambda \psi_{\ell} \rangle|^2 \notag\\
&=\sum\limits_{m\in\mathbb{Z}} \Big|\sqrt{2N} \int\limits_{0}^{N}\mathcal{S}((L^j\hat{f})\overline{H_{\ell}}\overline{\hat{\psi_0}})(\gamma)e^{2\pi i(2N)(2m)\gamma}\, d\gamma  \Big|^2 \notag \\
& \quad +\sum\limits_{m\in\mathbb{Z}} \Big|\sqrt{2N} \int\limits_{0}^{N}\mathcal{S}((L^j\hat{f})\overline{H_{\ell}}\overline{\hat{\psi_0}})(\gamma)e^{2\pi i(2N)(\frac{r}{N}+2m)\gamma}\, d\gamma  \Big|^2 \quad \text{(using \eqref{eq3.4rev})} \notag\\
&=\frac{1}{2}\sum\limits_{m\in\mathbb{Z}} \Big| \int\limits_{0}^{\frac{1}{4N}}(L^j\hat{f})(\gamma)\overline{H_{\ell}(\gamma)}\overline{\hat{\psi_0}(\gamma)}e^{2\pi i(4Nm)\gamma}2 \sqrt{N}\, d\gamma  \Big|^2 \notag\\
& \quad +\frac{1}{2}\sum\limits_{m\in\mathbb{Z}} \Big| \int\limits_{0}^{\frac{1}{4N}}(L^j\hat{f})(\gamma)\overline{H_{\ell}(\gamma)}\overline{\hat{\psi_0}(\gamma)}e^{2\pi i(2r)\gamma} e^{2 \pi i(4Nm)\gamma} 2 \sqrt{N}\, d\gamma  \Big|^2 \notag\\
&=\frac{1}{2}\int\limits_{0}^{\frac{1}{4N}}|(L^j\hat f)(\gamma)\overline{H_\ell(\gamma)}\overline{\hat{\psi_0}(\gamma)}|^2 d\gamma+\frac{1}{2}\int\limits_{0}^{\frac{1}{4N}}|(L^j\hat f)(\gamma)\overline{H_\ell(\gamma)}\overline{\hat{\psi_0}(\gamma)}|^2 d\gamma \notag\\
&=\int\limits_{0}^{\frac{1}{4N}}|(L^j\hat f)(\gamma)\overline{H_\ell(\gamma)}\overline{\hat{\psi_0}(\gamma)}|^2 d\gamma.
\end{align}

Since $H(\gamma)^\ast H(\gamma)=1$ a.e., so $H(\gamma)$ could be consider as an isometry from $\mathbb{C}^{1}$ into $\mathbb{C}^{n+1}$. Using \eqref{eq3.5extr}, we have
\begin{align}
\sum \limits_{\ell=0}^n\sum\limits_{\lambda \in \Lambda} |\langle f,L^{j-1}T_\lambda \psi_{\ell} \rangle|^2 & =\sum \limits_{\ell=0}^n   \int\limits_{0}^{\frac{1}{4N}}|(L^j\hat f)(\gamma)\overline{H_\ell(\gamma)}\overline{\hat{\psi_0}(\gamma)}|^2 d\gamma \notag\\
&=\displaystyle{\int \limits_{0}^{\frac{1}{4N}}\left\|\begin{bmatrix}
\overline{H_0(\gamma)}\notag\\
\vdots\\
%\dots\\
\overline{H_n(\gamma)}
\end{bmatrix}_{(n+1) \times 1} [(L^j\hat{f})\overline{\hat{\psi_0}}]_{1\times 1}\right\|^2_{\mathbb{C}^{n+1}}}d\gamma \notag\\
&=\int\limits_{0}^{\frac{1}{4N}}\left\|\overline{H(\gamma)}_{(n+1) \times1}[(L^j\hat{f})\overline{\hat{\psi_0}}]_{1\times 1}\right\|^2_{\mathbb{C}^{n+1}}d\gamma \notag\\
&=\int \limits_{0}^{\frac{1}{4N}}|(L^j\hat{f})(\gamma)\overline{\hat{\psi_0}}(\gamma)|^2\, d\gamma \label{A}.
\end{align}
Also
\begin{align}\label{eq3.99xv}
&\sum\limits_{\lambda \in \Lambda} |\langle f,L^{j}T_\lambda \psi_0 \rangle|^2 \notag\\
 &=\sum\limits_{\lambda \in 2\mathbb{Z}} |\langle f,L^{j}T_\lambda \psi_0 \rangle|^2+\sum\limits_{\lambda \in (\frac{r}{N}+2\mathbb{Z})} |\langle f,L^{j}T_\lambda \psi_0 \rangle|^2 \notag\\
&=\sum\limits_{m\in\mathbb{Z}} \Big| \int\limits_{\mathbb{R}}(L^j\hat{f})(\gamma)\overline{\hat{\psi_0}(\gamma)}e^{2\pi i(2m)\gamma}\, d\gamma  \Big|^2 +\sum\limits_{m\in\mathbb{Z}} \Big| \int\limits_{\mathbb{R}}(L^j\hat{f})(\gamma)\overline{\hat{\psi_0}(\gamma)}e^{2\pi i(\frac{r}{N}+2m)\gamma}\, d\gamma  \Big|^2.
\end{align}
Using Supp $\hat{\psi_0} \subseteq [0,\frac{1}{4N}]$, $\frac{1}{4N} <\frac{1}{2}$ and  applying the  Parseval formula on $L^2(0,\frac{1}{2})$ with respect to orthonormal basis $\{\sqrt{2}e^{2\pi i (2m)\gamma}\}_{m\in \mathbb{Z}}$, we compute
\begin{align}
&\sum\limits_{\lambda \in \Lambda} |\langle f,L^{j}T_\lambda \psi_0 \rangle|^2 \notag\\
&=\sum\limits_{m\in\mathbb{Z}} \Big| \int\limits_{0}^{N}\mathcal{S}((L^j\hat{f})\overline{\hat{\psi_0})}(\gamma)e^{2\pi i(2m)\gamma}\, d\gamma  \Big|^2 +\sum\limits_{m\in\mathbb{Z}} \Big| \int\limits_{0}^{N}\mathcal{S}((L^j\hat{f})\overline{\hat{\psi_0})}(\gamma)e^{2\pi i(\frac{r}{N}+2m)\gamma}\, d\gamma  \Big|^2 \quad \Big(\text{by} \ \eqref{eq3.99xv}\Big) \notag\\
&=\frac{1}{2}\sum\limits_{m\in\mathbb{Z}} \Big| \int\limits_{0}^{\frac{1}{2}}(L^j\hat{f})(\gamma)\overline{\hat{\psi_0}}(\gamma)\sqrt{2}e^{2\pi i(2m)\gamma}\, d\gamma  \Big|^2+\frac{1}{2} \sum\limits_{m\in\mathbb{Z}} \Big| \int\limits_{0}^{\frac{1}{2}}(L^j\hat{f})(\gamma)\overline{\hat{\psi_0}}(\gamma)\sqrt{2}e^{2\pi i(\frac{r}{N}+2m)\gamma}\, d\gamma  \Big|^2 \notag\\
&=\frac{1}{2}\int\limits_{0}^{\frac{1}{2}}|(L^j\hat f)(\gamma)\overline{\hat{\psi_0}(\gamma)} |^2d\gamma+\frac{1}{2}\int\limits_{0}^{\frac{1}{2}}|(L^j\hat f)(\gamma)\overline{\hat{\psi_0}(\gamma)} |^2d\gamma \notag\\
&=\int\limits_{0}^{\frac{1}{2}}|(L^j\hat f)(\gamma)\overline{\hat{\psi_0}(\gamma)} |^2d\gamma\notag\\
&=\int\limits_{0}^{\frac{1}{4N}}|(L^j\hat f)(\gamma)\overline{\hat{\psi_0}(\gamma)} |^2d\gamma. \label{B}
\end{align}
The proof now follows from  (\ref{A}) and (\ref{B}).
\end{proof}

\begin{lem}\label{lem2.3}
	Let $\{\psi_{\ell}, H_{\ell}\}_{\ell=0}^{n}$ be a nonuniform general setup, and let  $H(\gamma)^{\ast}H(\gamma)$=1. Then,  the following holds.
	\begin{enumerate}[$(i)$]
\item $\{T_\lambda \psi_0\}_{\lambda \in \Lambda} $ is Bessel sequence with Bessel bound $1$.
\item For any $f \in L^2(\mathbb{R})$,
\begin{align*}
\lim \limits_{j \to -\infty}\sum\limits_{\lambda \in \Lambda} |\langle f,L^jT_\lambda \psi_0 \rangle|^2=0.
\end{align*}
	\end{enumerate}
\end{lem}
\begin{proof} $(i):$  Let  $f \in L^2(\mathbb{R}) $ be  such that $\hat{f} \in C_c(\mathbb{R}) $, and let  $\epsilon >0$ be given. Then,  by Lemma \ref{lem2.1h}, we can find an integer $j>0$ such that
\begin{align}\label{eqcons18}
\sum\limits_{\lambda \in \Lambda} |\langle f,L^jT_\lambda \psi_0 \rangle|^2 \le (1+\epsilon)\|f\|^2.
\end{align}
Also, by Lemma \ref{lem2.2}, we have
\begin{align}\label{eqz}
\sum\limits_{\lambda \in \Lambda} |\langle f,L^{j-1}T_\lambda \psi_{\ell} \rangle|^2 \le \sum\limits_{\lambda \in \Lambda} |\langle f,L^jT_\lambda \psi_0 \rangle|^2.
\end{align}	
Applying \eqref{eqz} $j$ times and  using \eqref{eqcons18}, we get
\begin{align*}
\sum\limits_{\lambda \in \Lambda} |\langle f,T_\lambda \psi_0 \rangle|^2 \le
		\sum\limits_{\lambda \in \Lambda} |\langle f,L^jT_\lambda \psi_0 \rangle|^2 \le (1+\epsilon)\|f\|^2 .
\end{align*}
  Since $\epsilon >0$ was arbitrary, we have
\begin{align*}
\sum\limits_{\lambda \in \Lambda} |\langle f,T_\lambda \psi_0 \rangle|^2 \le\|f\|^2.
	\end{align*}
	
Because this inequality holds on a dense subset of $L^2(\mathbb{R})$, it holds on $L^2(\mathbb{R})$. This proves $(i)$.

\vspace{10pt}

$(ii)$:  Let $f\in L^2(\mathbb{R})$. Since $L^j$ is unitary map  for all $j \in \mathbb{Z}$, by using  $(i)$,  the family  $\{L^jT_\lambda \psi_0\}_{\lambda \in \Lambda} $ is Bessel sequence with Bessel bound $1$. For any  $j \in \mathbb{Z}$ and  for any bounded interval  $I \subset \mathbb{R}$, we have
\begin{align*}
\sum\limits_{\lambda \in \Lambda} |\langle f,L^jT_\lambda \psi_0 \rangle|^2 &\le 2
\sum\limits_{\lambda \in \Lambda} |\langle f\chi_{I},L^jT_\lambda \psi_0 \rangle|^2 + 2 \sum\limits_{\lambda \in \Lambda} |\langle f(1-\chi_{I}),L^jT_\lambda \psi_0 \rangle|^2\\
&\le 2
\sum\limits_{\lambda \in \Lambda} |\langle f\chi_{I},L^jT_\lambda \psi_0 \rangle|^2+2\|f(1-\chi_{I})\|^2.
\end{align*}
Now, $\|f(1-\chi_{I})\|^2 \to 0$,  if we choose $I$ to be sufficiently large. Therefore, we only need to show
\begin{align*}
\sum\limits_{\lambda \in \Lambda} |\langle f\chi_{I},L^jT_\lambda \psi_0 \rangle|^2 \rightarrow 0 \ \text{as} \  j  \rightarrow  -\infty.
\end{align*}
Using the Cauchy-Schwarz's inequality for integrals, we obtain
\begin{align}\label{eqldc1}
\sum\limits_{\lambda \in \Lambda} |\langle f\chi_{I},L^jT_\lambda \psi_0 \rangle|^2 &=(2N)^j \sum\limits_{\lambda \in \Lambda}\Big|\int \limits_{I} f(\gamma)\overline{\psi_0((2N)^j\gamma-\lambda)} \,d\gamma \Big|^2 \notag\\
&\le (2N)^j\|f\|^2 \sum \limits_{\lambda \in \Lambda} \int \limits_{I} |\psi_0((2N)^j\gamma-\lambda)|^2 \,d\gamma \notag\\
&=\|f\|^2 \sum \limits_{\lambda \in \Lambda} \displaystyle{\int \limits_{(2N)^jI-\lambda} |\psi_0(\gamma)|^2 \,d\gamma}.
\end{align}
Applying the Lebesgue  dominated convergence theorem in \eqref{eqldc1}, we have
\begin{align*}
\sum\limits_{\lambda \in \Lambda} |\langle f\chi_{I},L^jT_\lambda \psi_0 \rangle|^2 \rightarrow 0\  \text{as}\   j \to -\infty.
\end{align*}
Hence $(ii)$  is proved.
\end{proof}

\section{Unitary Extension Principle  for Nonuniform Wavelet Frames}  \label{sec4}

We begin  this section with the UEP for nonuniform wavelet frames for $L^2(\mathbb{R})$.
\begin{thm}\label{thm4.1}
	Let $\{\psi_{\ell}, H_{\ell}\}_{\ell=0}^{n}$ be  a nonuniform general setup and    $H(\gamma)^{\ast}H(\gamma)=1$. Then, the  nonuniform multiwavelets  system $\{L^jT_{\lambda}\psi_{\ell}\}_{j\in \mathbb{Z},\lambda \in \Lambda \atop \ell=1,2,\dots,n}$ constitutes a Parseval frame for $L^2(\mathbb{R})$.
\end{thm}
\begin{proof}
Let $\epsilon >0$ be given. Consider a function $f \in L^2(\mathbb{R})$ such that $\hat{f} \in C_c{(\mathbb{R})}$. By Lemma \ref{lem2.1h}, we can choose $J>0$ such that for all $j\geq J$,
\begin{align}\label{eq5g}
(1-\epsilon)\|f\|^2 \le\sum\limits_{\lambda \in \Lambda} |\langle f,L^jT_\lambda \psi_0 \rangle|^2 \le(1+\epsilon) \|f\|^2.
\end{align}
Using  Lemma \ref{lem2.2}, we have
\begin{align}\label{eq6g}
&\sum\limits_{\lambda \in \Lambda} |\langle f,L^jT_\lambda \psi_0 \rangle|^2 \notag\\
&=\sum \limits_{\ell=0}^n\sum\limits_{\lambda \in \Lambda} |\langle f,L^{j-1}T_\lambda \psi_{\ell} \rangle|^2 \notag\\
&=\sum\limits_{\lambda \in \Lambda} |\langle f,L^{j-1}T_\lambda \psi_0 \rangle|^2+\sum \limits_{\ell=1}^n\sum\limits_{\lambda \in \Lambda} |\langle f,L^{j-1}T_\lambda \psi_{\ell} \rangle|^2.
\end{align}
Applying  Lemma \ref{lem2.2} on $\sum\limits_{\lambda \in \Lambda} |\langle f,L^{j-1}T_\lambda \psi_0 \rangle|^2$, we get
\begin{align}\label{eq7g}
\sum\limits_{\lambda \in \Lambda} |\langle f,L^{j-1}T_\lambda \psi_0 \rangle|^2=\sum\limits_{\lambda \in \Lambda} |\langle f,L^{j-2}T_\lambda \psi_0 \rangle|^2+\sum \limits_{\ell=1}^n\sum\limits_{\lambda \in \Lambda} |\langle f,L^{j-2}T_\lambda \psi_{\ell} \rangle|^2.
\end{align}
By \eqref{eq6g} and \eqref{eq7g}, we have
\begin{align*}
\sum\limits_{\lambda \in \Lambda} |\langle f,L^jT_\lambda \psi_0 \rangle|^2=\sum\limits_{\lambda \in \Lambda} |\langle f,L^{j-2}T_\lambda \psi_0 \rangle|^2+\sum\limits_{\ell=1}^{n}\sum\limits_{\lambda \in \Lambda}\sum\limits_{p=j-2}^{j-1}| \langle f,L^pT_\lambda \psi_{\ell} \rangle|^2.
\end{align*}
Repeating the above arguments, for any $m <j$, we have
\begin{align}\label{eq8g}
\sum\limits_{\lambda \in \Lambda} |\langle f,L^jT_\lambda \psi_0 \rangle|^2=\sum\limits_{\lambda \in \Lambda} |\langle f,L^{m}T_\lambda \psi_0 \rangle|^2+\sum\limits_{\ell=1}^{n}\sum\limits_{\lambda \in \Lambda}\sum\limits_{p=m}^{j-1}| \langle f,L^pT_\lambda \psi_{\ell} \rangle|^2.
\end{align}
It follows from  \eqref{eq5g} and \eqref{eq8g} that  for all $j\ge J,m<j$,
\begin{align*}
(1-\epsilon)\|f\|^2 \le \sum\limits_{\lambda \in \Lambda} |\langle f,L^{m}T_\lambda \psi_0 \rangle|^2+\sum\limits_{\ell=1}^{n}\sum\limits_{\lambda \in \Lambda}\sum\limits_{p=m}^{j-1}| \langle f,L^pT_\lambda \psi_{\ell} \rangle|^2 \le (1+\epsilon)\|f\|^2.
\end{align*}
Letting $m \to -\infty$ in above and using  $(ii)$  of Lemma \ref{lem2.3}, we have
\begin{align}\label{eq1ins}
(1-\epsilon)\|f\|^2 \le \sum\limits_{\ell=1}^{n}\sum\limits_{\lambda \in \Lambda}\sum\limits_{p=-\infty}^{j-1}| \langle f,L^pT_\lambda \psi_{\ell} \rangle|^2 \le (1+\epsilon)\|f\|^2.
\end{align}
Letting $j \to \infty$  in \eqref{eq1ins}, we have
\begin{align*}
(1-\epsilon)\|f\|^2 \le \sum\limits_{\ell=1}^{n}\sum\limits_{\lambda \in \Lambda}\sum\limits_{p=-\infty}^{\infty}| \langle f,L^pT_\lambda \psi_{\ell} \rangle|^2 \le (1+\epsilon)\|f\|^2.
\end{align*}
Since $\epsilon >0$ was arbitrary, we obtain
\begin{align*}
\sum\limits_{\ell=1}^{n}\sum\limits_{\lambda \in \Lambda}\sum\limits_{p \in \mathbb{Z}}| \langle f,L^pT_\lambda \psi_{\ell} \rangle|^2=\|f\|^2 \ \text{for all} \ f \in L^{2}(\mathbb{R}),
\end{align*}
as desired.
\end{proof}

The next theorem gives  the generalized (or oblique) extension principle  for nonuniform wavelet frames in $L^{2}(\mathbb{R})$. It gives the more flexible technique to construct nonuniform wavelet frames.
%For the applications and other technical details about  the oblique extension principle for standard wavelet frames, we refer to \cite[p. 460]{OC2} and  \cite[Proposition 1.11]{DHRS}.
\begin{thm}\label{thm4.2}
	Let $\{\psi_{\ell}, H_{\ell}\}_{\ell=0}^n$ be a nonuniform general setup. Assume that  there exist strictly positive function $\theta \in L^{\infty}(\mathbb{R})$ for which
\begin{align*}
&\lim\limits_{\gamma \to 0^{+}} \theta(\gamma)=1,
\intertext{and}
&\theta(2N\gamma)|H_0(\gamma)|^2+\sum\limits_{\ell=1}^{n}|H_{\ell}(\gamma)|^2=\theta(\gamma).
\end{align*}
Then, $\{L^jT_{\lambda}\psi_{\ell}\}_{j\in \mathbb{Z},\lambda \in \Lambda \atop  \ell=1,2,\cdots,n}$ is  a Parseval nonuniform wavelet frame for $L^2(\mathbb{R})$.
	\end{thm}

\begin{proof}
Define $\widetilde{\psi_0}\in L^2(\mathbb{R})$ such that
\begin{align}\label{eqob1}
\widehat{\widetilde{\psi_0}}(\gamma)=\sqrt{\theta (\gamma)}\hat{\psi_0} (\gamma).
\end{align}
Define functions $\widetilde{H_0},\widetilde{H_1},\dots,\widetilde{H_n}$ as follows
\begin{align*}
\widetilde{H_0}(\gamma)&=\sqrt{\frac{\theta(2N\gamma)}{\theta(\gamma)}}H_0(\gamma),\\
\widetilde{H_{\ell}}(\gamma)&=\sqrt{\frac{1}{\theta(\gamma)}}H_{\ell}(\gamma),\   \ell=1,2,\dots,n.
\end{align*}
Then, we have
\begin{align}
\widehat{\widetilde{\psi_0}}(2N\gamma)&=\sqrt{\theta(2N\gamma)}\hat{\psi_0}(2N\gamma)\notag \\
&=\sqrt{\theta(2N\gamma)}  H_0(\gamma)\hat{\psi_0}(\gamma)\notag\\
&=\sqrt{\theta(2N\gamma)}\left(H_0(\gamma)\frac{\widehat{\widetilde{\psi_0}}(\gamma)}{\sqrt{\theta(\gamma)}}\right)\notag\\
&=\sqrt{\frac{\theta(2N\gamma)}{\theta(\gamma)}}H_0(\gamma)\widehat{\widetilde{\psi_0}}(\gamma)\notag\\
&=\widetilde{H_0}(\gamma)\widehat{\widetilde{\psi_0}}(\gamma) \label{z},
\end{align}
and
\begin{align}
\lim\limits_{\gamma \to 0^{+}} \widehat{\widetilde{\psi_0}}(\gamma)=\lim\limits_{\gamma \to 0^{+}}\sqrt{\theta(\gamma)}\hat{\psi_0}(\gamma)= 1.\label{y}
\end{align}
Since $\{\psi_{\ell}, H_{\ell}\}_{\ell=0}^n$ is a nonuniform general setup, by \eqref{eqob1}, we have
\begin{align}
\text{Supp}\   \widehat{\widetilde{\psi_0}}(\gamma) \subseteq \Big[0, \frac{1}{4N}\Big], \label{zz}
\end{align}
and
\begin{align}
\sum\limits_{\ell=0}^{n}|\widetilde{H_{\ell}}(\gamma)|^2&=|\widetilde{H_0}(\gamma)|^2+\sum\limits_{\ell=1}^{n}|\widetilde{H_{\ell}}(\gamma)|^2\notag
\\
&=\frac{\theta(2N\gamma)}{\theta(\gamma)}|H_0(\gamma)|^2+\sum\limits_{\ell=1}^{n}\frac{|H_{\ell}(\gamma)|^2}{\theta(\gamma)}\notag\\
&=\frac{1} {\theta(\gamma)}\theta(\gamma)\notag\\
&=1.\label{q}
\end{align}
Thus
\begin{align}
\widetilde{H_{\ell}}(\gamma)\in L^{\infty}(\mathbb{R}) \ \text{for} \  \ell =0,1,\dots,n. \label{qq}
\end{align}
Let  $\widetilde{\psi}_1, \widetilde{\psi}_2,\dots, \widetilde{\psi}_n \in L^2(\mathbb{R})$ be such that
\begin{align}
\widehat{\widetilde{\psi}}_{\ell}(2N\gamma)= \widetilde{H}_{\ell} (\gamma)\widehat{\widetilde{\psi_0}}(\gamma), \  \ell = 1,\dots,n.
\end{align}\label{x}
Define

\begin{align*}
\widetilde{H}(\gamma)=\begin{bmatrix}
\widetilde{{H_0}}(\gamma)\\
\widetilde{{H_1}}(\gamma)\\
\vdots\\
\widetilde{H_n}(\gamma)
\end{bmatrix}_{(n+1) \times 1}.
\end{align*}
Then,  by \eqref{z}, \eqref{y}, \eqref{zz} and \eqref{qq},  the collection $\{\widetilde{\psi_{\ell}}, \widetilde{H_{\ell}}\}_{\ell=0}^n$ is  a nonuniform general setup.

Using \eqref{q}, we have
\begin{align*}
\widetilde{H}(\gamma)^\ast \widetilde{H}(\gamma)
&=\sum\limits_{\ell=0}^{n}|\widetilde{H_{\ell}}(\gamma)|^2=1.
\end{align*}
Hence, by Theorem \ref{thm4.1}, $\{L^jT_{\lambda}\widetilde{\psi}_{\ell}\}_{j\in \mathbb{Z},\lambda \in \Lambda \atop \ell =1,2,\dots,n}$ is a Parseval nonuniform wavelet frames for $L^2(\mathbb{R})$.

Next, we compute
\begin{align*}
\hat{\psi_{\ell}}(2N\gamma)&=H_{\ell}(\gamma)\hat{\psi_0}(\gamma)\\
&=\left(\widetilde{H_{\ell}}(\gamma)\sqrt{\theta(\gamma)}\right)\left(\frac{\widehat{\widetilde{\psi_0}}(\gamma)}{\sqrt{\theta(\gamma)}}\right)\\
&=\widetilde{H_{\ell}}(\gamma)\widehat{\widetilde{\psi_0}}(\gamma)\\
&=\widehat{\widetilde{\psi_{\ell}}}(2N\gamma).
\end{align*}
This gives,  $\psi_{\ell}=\widetilde{\psi_{\ell}}$. Hence, the system  $\{L^jT_{\lambda}\psi_{\ell}\}_{j\in \mathbb{Z},\lambda \in \Lambda \atop \ell =1,2,\cdots,n}$ is a Parseval nonuniform wavelet frames for $L^2(\mathbb{R})$.
\end{proof}
\begin{rem}
It is worth noticing that, when $\theta=1$, Theorem \ref{thm4.1} can be obtained from Theorem \ref{thm4.2}.
\end{rem}
\textbf{Construction of nonuniform wavelet frame with two generators: }
Computational effort reduces if we have less number of generator or window functions, so we wish to have as minimum numbers of generators as is it possible. In this direction, we have the following result as   an application of Theorem \ref{thm4.2}.
\begin{cor}\label{cor4.3}

 Let $\psi_0 \in L^2(\mathbb{R})$ such that
\begin{enumerate}[$(i)$]
\item $\hat{\psi_0}(2N\gamma)=H_0(\gamma)\hat{\psi_0}(\gamma)$, where $H_0(\gamma)\in L^{\infty}(\mathbb{R})$;
\item Supp $\hat{\psi_0}(\gamma) \subseteq [0,\frac{1}{4N}]$; and
\item $\lim \limits_{\gamma \to 0^+}\hat{\psi_0}(\gamma)=1$.
\end{enumerate}
If we choose $H_1(\gamma)=\sqrt{\theta(2N\gamma)}H_0(\gamma)i, \ H_2(\gamma)=\sqrt{\theta(\gamma)}$,
and  $\psi_1,\psi_2 \in L^2(\mathbb{R})$ such that
\begin{align*}
\hat{\psi_\ell}(2N\gamma)=H_\ell (\gamma)\hat{\psi_0}(\gamma), \ \ell=1,2.
\end{align*}
Then 	
\begin{align*}
\theta(2N\gamma)|H_0(\gamma)|^2+|H_1(\gamma)|^2+|H_2(\gamma)|^2  =\theta(\gamma).
\end{align*}
Hence, by Theorem \ref{thm4.2}, $\{L^jT_{\lambda}\psi_{\ell}\}_{j\in \mathbb{Z},\lambda \in \Lambda \atop \ell=1,2}$ form a Parseval nonuniform wavelet frame for $L^2(\mathbb{R})$.
\end{cor}

\section{Examples}   \label{sec5}

This section gives some applicative examples of the  UEP and its generalized version. The following example  illustrates Theorem  \ref{thm4.1}.

\begin{exa}
	Let $ N=2,r=3 $, and  $\psi_0 \in L^2{(\mathbb{R})}$ be such that
\begin{align*}
\hat{\psi_0}(\gamma)=\frac{\sin (\gamma)}{\gamma}\chi_{]0,\frac{1}{8}]}(\gamma).
\end{align*}
Then
\begin{align*}
&(i) \ \lim\limits_{\gamma \to 0^+}\hat{\psi_0}(\gamma)=1;\\
&(ii)\  \text{Supp} \  \hat{\psi_0} \subseteq [0,\frac{1}{8}]; \  \text{and} \\
& (iii) \ \hat{\psi_0}(4\gamma)=\frac{\sin(4\gamma)}{4\gamma}\chi_{]0,\frac{1}{8}]}(4\gamma)\\
&\quad \quad \quad \quad \quad=\frac{4\sin(\gamma)\cos(\gamma)\cos(2\gamma)}{4\gamma}\chi_{]0,\frac{1}{32}]}(\gamma)\chi_{]0,\frac{1}{8}]}(\gamma)\\
& \quad \quad \quad \quad \quad = H_0(\gamma)\hat{\psi_0}(\gamma),
\end{align*}

where $H_0(\gamma)=\cos(\gamma)\cos(2\gamma)\chi_{]0,\frac{1}{32}]}(\gamma)$.

Let \begin{align*}
H_1(\gamma)&=\cos(2\gamma)\sin(\gamma)\chi_{]0,\frac{1}{32}]}(\gamma),\\
H_2(\gamma)&=\sin(2\gamma)\chi_{]0,\frac{1}{32}]}(\gamma), \ \text{and}\\
H_3(\gamma)&=\chi_{\mathbb{R}\setminus ]0,\frac{1}{32}]}(\gamma).
\end{align*}
Let  $\psi_1,\psi_2,\psi_3 \in L^2({\mathbb{R}})$ be such that
\begin{align*}
\hat{\psi}_{\ell}(4\gamma)=H_{\ell}(\gamma)\hat{\psi_0}(\gamma),\  \ell=1,2,3.
\end{align*}
Choose
\begin{align*}
 H(\gamma)=\begin{bmatrix}
H_0(\gamma)\\
H_1(\gamma)\\
H_2(\gamma)\\
H_3(\gamma)
\end{bmatrix}.
\end{align*}
Then, $\{\psi_\ell,H_\ell\}_{\ell=0}^3$ is a nonuniform general setup such that
\begin{align*}
  {H(\gamma)}^\ast H(\gamma)=|H_0(\gamma)|^2+|H_1(\gamma)|^2+|H_2(\gamma)|^2+|H_3(\gamma)|^2=1.
\end{align*}

Hence, by Theorem \ref{thm4.1},
$\{L^jT_{\lambda}\psi_{\ell}\}_{j\in \mathbb{Z},\lambda \in \{0,\frac{3}{2}\}+2\mathbb{Z} \atop \ell=1,2,3}$ is a  nonuniform Parseval wavelet frame $L^{2}(\mathbb{R})$.
\end{exa}

To conclude the paper, we illustrate Theorem \ref{thm4.2} with the following example.
\begin{exa}
Let $N=2$, $r=3$ and $\psi_0\in L^2({\mathbb{R}})$ be such that
\begin{align*}
\hat{\psi_0}(\gamma)=\chi_{[0,\frac{1}{8}]}(\gamma).
\end{align*}
Then
\begin{align*}
&(i) \ \lim \limits_{\gamma \to 0^+}\hat{\psi_0}(\gamma)=1; \\
&(ii)  \  \text{Supp} \  \hat{\psi_0}(\gamma)\subseteq [0,\frac{1}{8}]; \ \text{and} \\
&(iii) \ \psi_0(4\gamma) =\chi_{[0,\frac{1}{8}]}(4\gamma)\\
&\quad \quad \quad \quad \quad =\chi_{[0,\frac{1}{32}]}(\gamma)\chi_{[0,\frac{1}{8}]}(\gamma)\\
& \quad \quad \quad \quad \quad  =H_0(\gamma)\hat{\psi_0}(\gamma),
\end{align*}
where $H_0(\gamma)=\chi_{[0,\frac{1}{32}]}(\gamma)\in L^{\infty}(\mathbb{R})$.

\vspace{10pt}

Let $\theta(\gamma)=1$ and define $H_1(\gamma)=\chi_{\mathbb{R}\setminus[0,\frac{1}{32}]}$

Then, the collection  $\{\psi_\ell,H_\ell\}_{\ell=0}^1$ is a nonuniform general setup such that
\begin{align*}
\theta(4\gamma)|H_0(\gamma)|^2+|H_1(\gamma)|^2=\theta(\gamma).
\end{align*}
Hence, by Theorem \ref{thm4.2}, the nonuniform wavelet system $\{L^jT_{\lambda}\psi_{1}\}_{j\in \mathbb{Z},\lambda \in \{0,\frac{3}{2}\}+2\mathbb{Z}} $ is a  Parseval  frame for  $L^{2}(\mathbb{R})$.
\end{exa}


\begin{thebibliography}{99}\baselineskip8pt
 \bibitem{Beno1}
J. ~Benedetto and  O.~Treiber, Wavelet frames: multiresolution analysis
and extension principles. In: "Wavelet transforms and time-frequency signal
analysis", 1-36,  Birkh$\ddot{a}$user, Boston, 2001.

 \bibitem{BN1}
A. ~Boggess and F.~ J.~ Narcowich, A First Course in Wavelets with Fourier Analysis,  John Wiley $\&$ Sons, Inc., Hoboken, NJ, 2009.

 \bibitem{CK}
P. G. Casazza and G. Kutyniok, Finite frames: Theory and Applications, Birkh$\ddot{a}$user, 2012.

\bibitem{CKKm}
 O.~ Christensen, H.~O.~Kim and R.~Y.~Kim, Extensions of Bessel sequences to dual pairs of frames,
\emph{Appl. Comput. Harmon. Anal.},  34 (2)(2013),  224--233.

\bibitem{OC2}
 O. Christensen, An introduction to frames and Riesz bases, Second edition,  Birkh$\ddot{a}$user, 2016.

 \bibitem{OCG1}
  O.~ Christensen and S. ~ S.~ Goh, The unitary extension principle on locally compact abelian groups, \emph{Appl. Comput. Harmon.Anal.}, to appear. Available at http://dx.doi.org/10.1016/j.acha.2017.07.004.

\bibitem{ID}
I. ~ Daubechies, Ten Lectures on Wavelets, SIAM, Philadelphia, 1992.


\bibitem{DHRS}
I. ~ Daubechies, B.~ Han, A.~Ron and Z.~Shen, Framelets: MRA-based constructions of wavelet frames,
\emph{Appl. Comput. Harmon. Anal.}, 14 (1) (2003),  1--46.


\bibitem{DLKV}
Deepshikha and L. ~ K. Vashisht, A note on discrete frames of translates in $\mathbb{C}^N$,  \emph{TWMS  J.  Appl. Eng. Math.},  6 (1) (2016), 143--149.


\bibitem{DV2}
Deepshikha and  L. ~K. Vashisht, Necessary and sufficient conditions for
discrete wavelet frames in $\mathbb{C}^N$, \emph{J. Geom. Phys.}, 117 (2017), 134--143.

\bibitem{D}
Dao-Xin Ding, Generalized continuous frames constructed by using an iterated function system, \emph{J. Geom. Phys.},  61 (2011) 1045--1050.



\bibitem{DS}
R. ~J. Duffin and A.~ C. Schaeffer, A class of nonharmonic Fourier series, \emph{Trans. Amer. Math.
Soc.}, 72 (1952), 341--366.


\bibitem{GN1}
J.~ P. Gabardo and M. ~Z. ~ Nashed,  Nonuniform multiresolution analysis and spectral pairs, \emph{J. Funct.  Anal.}, 158  (1998), 209--241.

\bibitem{GN2}
J.~ P. Gabardo and X.~ Yu, Wavelets associated with nonuniform multiresolution analyses and one-dimensional spectral pairs, \emph{J. Math. Anal. Appl.},  323 (2006),  798--817.

\bibitem{Bhan}
 B.~ Han, Framelets and Wavelets: Algorithms, Analysis, and Applications, Birkh$\ddot{a}$user, 2017.

\bibitem{Hern1}
 E.~ Hernandez and G.~Weiss, A First Course on Wavelets. CRC Press, Boca
Raton, 1996.


\bibitem{Ron1}
A.~ Ron and  Z.~ Shen, Affine systems in $L2(\mathbb{R}^d)$: the analysis of the analysis operator, \emph{J. Funct. Anal.}, 148 (1997),  408--447.

\bibitem{RBCDR}
 M.~ B. ~ Ruskai, G.~ Beylkin, R.~ Coifman, I.~Daubechies, S. Mallat, Y.~ Meyer and L.~ Raphael, Wavelets and Their Applications,  Jones and Bartlett Publishers, Boston, MA, 1992.

\bibitem{SPman}
V. ~Sharma and P. ~Manchanda, Nonuniform wavelet frames in $L^2(\mathbb{R})$, \emph{Asian-Eur. J. Math.}, 8 (2) (2015), 1550034, 15 pp.


\bibitem{YuG}
X. ~Yu and J.~ P.~ Gabardo, Nonuniform wavelets and wavelet sets related to one-dimensional spectral pairs, \emph{J. Approx. Theory},  145 (2007) (1), 133--139.


\bibitem{VD3}
L. ~K. Vashisht and Deepshikha, Weaving properties of generalized continuous frames generated  by an iterated
function system,  \emph{J. Geom. Phys.}, 110 (2016), 282--295.

\bibitem{RZ1}
R. ~A. Zalik, Riesz bases and multiresolution analyses,
\emph{Appl. Comput. Harmon. Anal.}, 7  (3) (1999),  315--331.



\bibitem{RZ3}
R. ~A. Zalik, Orthonormal wavelet systems and multiresolution analyses, \emph{J. Appl. Funct. Anal.},  5 (1) (2010),  31--41.


\end{thebibliography}
\end{document}